%% file: defEquivariantVB.tex
\documentclass[11pt]{amsart}
\usepackage{amsmath,amsfonts,amssymb,amsthm,mathtools,multirow,tabls,mathrsfs}
\usepackage[margin=1.0in]{geometry}   %
\usepackage{lmodern}
\usepackage[all]{xy}
\usepackage[utf8]{inputenc}
\usepackage[T1]{fontenc}
\usepackage[shortlabels]{enumitem}
\usepackage[colorlinks,citecolor=magenta,linkcolor=blue]{hyperref}
\numberwithin{equation}{subsection}

\input{macrosMath}

\author[M. Zdanowicz]{Maciej Emilian Zdanowicz}
\address{\'Ecole Polytechnique F\'ed\'erale de Lausanne, Chair of Algebraic Geometry \newline 
  \indent MA C3 585 (Bâtiment MA), Station 8, CH-1015 Lausanne}
\email{maciej.zdanowicz@epfl.ch}

\title[Rigid schemes and equivariant bundles]{Arithmetically rigid schemes via deformation theory of equivariant vector bundles} \date{\today}

\begin{document}

\begin{abstract}
We analyze the deformation theory of equivariant vector bundles.  In particular, we provide an effective criterion for verifying whether all infinitesimal deformations preserve the equivariant structure.  As an application, using rigidity of the Frobenius homomorphism of general linear groups, we prove that projectivizations of Frobenius pullbacks of tautological vector bundles on Grassmanians are arithmetically rigid, that is, do not lift over rings where $p \neq 0$.  This gives the same conclusion for Totaro's examples of Fano varieties violating Kodaira vanishing.  We also provide an alternative purely geometric proof of non-liftability mod $p^2$ and to characteristic zero of the Frobenius homomorphism of a reductive group of non-exceptional type.
\end{abstract} 

\maketitle


\section{Introduction}

The following paper is motivated by the will to understand deformation theoretic phenomena arising in characteristic $p>0$ algebraic geometry.  By the classical result of Deligne and Illusie \cite{deligne_illusie}, a natural source of interesting examples is given by varieties violating Kodaira vanishing.  More precisely, the authors prove that a smooth variety defined over a perfect field $k$ of characteristic $p>0$ does not lift over the ring of Witt vectors of length two (\emph{mod $p^2$} for short) if $\dim X < p$ and $X$ admits an ample line bundle $\cL$ such that $H^i(X,\omega_X \otimes \cL) \neq 0$, for some $i > 0$. 

In a recent paper \cite{totaro_failure}, Totaro constructed varieties violating Kodaira vanishing satisfying two additional peculiar properties.  Firstly, they are very specific geometrically -- the anticanonical bundle is ample.  Secondly, the dimension condition $\dim X < p$ is not satisfied, and therefore the deformation theoretic behaviour cannot be described using merely the results of \cite{deligne_illusie}.  The main goal of this paper is to develop appropriate tools to overcome this problem and derive some interesting consequences in deformation theory.

Let us first present the main technique in our toolkit.  The Totaro's varieties, described thoroughly in \S\ref{ss:intro_examples_totaro}, are constructed as projectivizations of Frobenius pullbacks of equivariant vector bundles on proper homogeneous spaces for the action of a reductive group.  Such bundles arise from representations of the corresponding parabolic subgroup and therefore it is natural to ask when the deformations of vector bundles preserve the equivariant structure, and consequently give rise to deformations of the representation.  The following general result gives a necessary criterion for this to happen. 
\begin{thm}{(Theorem~\ref{thm:def_equiv_vb})}
\label{thm:_intro_def_equiv_vb}
Let $\tilde G$ be a reductive group scheme over a complete local ring $R$ with residue field $k$.  Let $\tilde X/R$ be a scheme equipped with a $\tilde G$-action, and let $\cE$ be a $G$-equivariant vector bundle on $X = \tilde X \otimes k$.  Then the natural forgetful transformation of deformation functors $\Def^G_{\cE} \to \Def_{\cE}$ is smooth if the following conditions are satisfied:
\begin{enumerate}[i)]
   \item $H^1(X,\cEnd(\cE))$ is trivial as a $G$-module,
   \item $H^2\left(G,H^0(X,\cEnd(\cE))\right) = 0$.
 \end{enumerate} 
\end{thm}

\noindent The conditions above have a natural interpretation.  The first one expresses the requirement that every $k[\epsilon]/\epsilon^2$-deformation, i.e., an element of the tangent space of the non-equivariant deformation functor, corresponding to an element of $H^1(X,\cEnd(\cE))$, needs to be preserved by the $G$-action.  The second condition concerns the obstruction classes to lifting the cocycle giving the equivariant structure (cf.~Remark~\ref{rem:obs_descr}). 

\subsection{Applications -- deformation theory of Totaro's examples}
\label{ss:intro_examples_totaro}

Our main application of the above result is to deformation theory of Totaro's varieties.  Let us now recall Totaro's construction.  From now on, by $k$ we denote an algebraically closed field of characteristic $p \geq 5$, and by $W(k)$ the associated ring of Witt vectors.  Let $N = p+2$, and let ${\pi \colon {\rm Fl}(1,2,N) \to {\rm Gr}(2,N)}$ be the natural projection from the partial flag variety to the Grassmanian of two-dimensional subspaces ${\rm Gr}(2,N)$.  The variety $X$ is defined by the Frobenius pullback of $\pi$, that is, by the cartesian diagram:
\[
\xymatrix{
  X \ar[r]\ar[d]\ar@{}[dr]|-{\square} & {\rm Fl}(1,2,N) \ar[d]^{\pi} \\
  {\rm Gr}(2,N) \ar[r]_(0.45){F} & {\rm Gr}(2,N).
  }
\]

\noindent Since ${\rm Fl}(1,2,N)$ is isomorphic to the projectivization $\PP_{{\rm Gr}(2,N)}(\cS)$ of the tautological vector bundle $\cS$ on ${\rm Gr}(2,N)$, the variety $X$ is in fact isomorphic to $\PP_{{\rm Gr}(2,N)}(F^*\cS)$.  By \cite[Theorem 2.1]{totaro_failure} we know that there exists a very ample divisor $A$ on $X$ satisfying the following properties:
\begin{enumerate}[i)]
  \item $-K_X = 2A$, \label{it1}
  \item $\chi(X,\cO_X(K_X + 3A)) < 0$, \label{it2}
  \item $H^i(X,\cO_X(A)) = 0$ for $i \geq 2$. \label{it3}
\end{enumerate}
In particular, the variety $X$ is Fano and violates the Kodaira vanishing theorem.  

We approach the deformation theory of Totaro's varieties by treating a more general problem of understanding the deformations of Frobenius pullbacks of tautological bundles on Grassmanians ${\rm Gr}(d,n)$, for $2 \leq d \leq n-2$.  For this purpose, we first use a simple argument to show that deformations of the projectivization in fact induces a deformation of the vector bundle.  We are then in the position to apply Theorem~\ref{thm:_intro_def_equiv_vb}.  Since we are working in characteristic $p>0$, the representation theory of reductive groups is fairly complicated, and therefore we need to apply the Borel--Bott vanishing substitutes proven by Andersen (see \S\ref{ss:cohomology_representation_thy} for the details).  In the end we see that deformations of the varieties in question induce deformations of the associated representation of the parabolic subgroup.  It turns out that the presence of the Frobenius pullback in the definition of the vector bundle implies that associated representation of the parabolic subgroup factors through the Frobenius homomorphism.  Consequently every deformation of the bundle yields a deformation of the Frobenius homomorphism of $\GL_d$, for $d \geq 2$.  The results described in the next section imply that the Frobenius of a reductive group of positive rank deforms only to rings where $p = 0$, and hence we obtain the following:

\begin{thm}[{Theorem~\ref{thm:defos_flag_varieties}}]
\label{thm:intro_defos_flag_varieties}
Let $\cS$ be the tautological vector bundle on the Grassmanian variety $\Gr(d,N)$, for $2 \leq d \leq N - 2$.  Then $\PP(F^*\cS)$ does not lift to any ring where $p \neq 0$.  In particular, the same conclusion holds for Totaro's examples.
\end{thm}

\noindent This describes the deformation theory of Totaro's examples, and moreover provides a bunch of new examples of arithmetically rigid schemes.  Such varieties are very rare.  First examples were constructed by Serre \cite[\S 8.6 and \S 8.7]{FGA}.  As proven by Ekedahl \cite{Ekedahl}, other examples are given by non-liftable Calabi--Yau varieties in \cite{Hirokado,Schroer}.  Furthermore, our results imply that the claim in Corollary 0.3 of \cite{bhatt} is true for every integer $n \geq 4$.



\subsection{Deformations of Frobenius -- another approach}
\label{ss:intro_description_other_approach}

An essential part of the above argument was the rigidity of the Frobenius homomorphism of a reductive group over a field of characteristic $p>0$.  This property can be derived from an intricate theory of reductive groups over arbitrary bases developed in \cite{SGA3III} (see \S\ref{sec:frobenius_group_schemes}).  We also note that non-liftability fo the Frobenius homomorphism of $\GL_N$ over $W$ was proven by Buium \cite[Corollary 4.116]{buium2} using his theory of $p$-differentials.  In the last part of this paper, we provide an alternative approach to this problem and prove that the Frobenius homomorphism of a reductive group of non-exceptional type lifts neither mod $p^2$ nor to characteristic zero.  We decided to include the argument in the paper, because it is of purely geometric nature and provide an insight which might be useful for considerations concerning other potentially non-reductive groups.  

In informal terms, our main observation is that a mod $p^2$ (resp. characteristic zero) lifting of the Frobenius homomorphism of a group scheme $G/k$ gives a natural lifting of the Frobenius pullback of every mod $p^2$ (resp. characteristic zero) liftable principal bundle of $G$.  For the general linear group $\GL_2$, identifying its principal bundles with rank two vector bundles, we obtain a contradiction with the results of Lauritzen--Rao \cite{lauritzen_rao} described in detail in Example~\ref{example:lauritzen_rao} (resp. with the results of Totaro).  


\subsection{Structure of the paper}

The paper is organized as follows.  In \S\ref{sec:preliminaries} we provide some preliminary results concerning main objects of our considerations.  Then, in \S\ref{sec:equivariant_sheaves} we analyse deformation theory of equivariant bundles, and in particular prove Theorem~\ref{thm:_intro_def_equiv_vb}.  Subsequently, in \S\ref{sec:preliminaries_reductive_groups} and \S\ref{sec:frobenius_group_schemes} we recall the general results concerning reductive group schemes over arbitrary bases, and then apply them to the deformation theory of Frobenius homomorphisms of reductive group schemes.  In \S\ref{sec:applications_totaro}, we combine previous results to prove Theorem~\ref{thm:intro_defos_flag_varieties}.  Finally, \S\ref{sec:another_proof_non_liftability} contains an alternative geometric proof of non-liftability of the Frobenius homomorphism of reductive groups.


\subsection{Acknowledgements}

I would like to thank Piotr Achinger, Joachim Jelisiejew, Łukasz Sien- kiewicz and Burt Totaro for helpful discussions concerning the content of the paper and many suggestions concerning the presentation.  Moreover, I am grateful to Prof. Alexandru Buium for his prompt reply to my e-mail question and a useful reference.  The work was supported by Zsolt Patakfalvi's Swiss National Science Foundation Grant No. $200021 / 169639$.


\subsection{Notation}

Throughout the paper, if not stated otherwise, $k$ is a perfect field of characteristic $p>0$ and $W(k)$ is the ring of Witt vectors of $k$.  We denote by $\Art_{W(k)}(k)$ the category of Artinian local $W(k)$-algebras with residue field $k$.  A surjection of rings in $\Art_{W(k)}(k)$ is called a \emph{small extension} if its kernel is of square zero.  For a $G$-action on a scheme $X$, by $[X/G]$ we denote the associated fppf quotient stack.  We also use the notation $BG_{S}$ for the special case of the classifying space $[S/G]$ of a group scheme $G$ over $S$, often omitting the subscript if it is clear from the context.  For any scheme $S$ by $\Sch_S$, we denote the category of $S$-schemes.


\section{Preliminaries}
\label{sec:preliminaries}

In this section, we present some preliminary results concerning the main subjects of our interest: Frobenius morphism, non-equicharacteristic deformation theory and group actions.  We also recall some basic properties of quotient stacks, since they are a handy tool in some of the following considerations.


\subsection{Frobenius morphism}
\label{sub:preliminaries_frobenius_and_deformations}

Let $p > 0$ be a prime number.  For every scheme $X$ defined over $\FF_p$ we consider the Frobenius morphism $F_X \colon X \to X$ defined as the identity on the level of topological spaces and the map of sheaves of rings $F_X^{\#} \colon \cO_X \to F_*\cO_X$ given by $f \mapsto f^p$.  For every morphism $\pi \colon X \to S$ of schemes over $\FF_p$, the associated Frobenius morphisms are compatible with $\pi$, which in turn yields a diagram

\[
\xymatrix{
  X \ar[rd]^{F_{X/S}}\ar@/_1.2pc/[rdd]_{\pi} \ar@/^1.2pc/[rrd]^{F_X}& & \\
   & X' \ar[r]\ar[d]_{\pi'}\ar@{}[dr]|-{\square} & X \ar[d]^{\pi} \\
   &  S \ar[r]_{F_S} & S,
}
\]
defining the Frobenius twist $\pi \colon X' \to S$ and the relative Frobenius morphism $F_{X/S} \colon X \to X'$.  We note that if $X$ is a group scheme over $S$ then the associated Frobenius morphism is in fact a homomorphism of groups schemes.


\subsection{Basics of deformation theory}
\label{sub:preliminaries_deformations}

Here, we recall the necessary tools from deformation theory.  For a general treatment of the topic we refer to \cite{schlessinger,hartshorne_deformation}.  
\begin{defin}
Let $X$ be a scheme over a perfect field $k$ of characteristic $p > 0$, and let $A \in \Art_{W(k)}(k)$ be a local Artinian $W(k)$-algebra with residue field $k$.  We say that a morphism $\pi \colon X_A \to \Spec A$ is an $A$-deformation (or an $A$-lifting if $p \neq 0$ in $A$) of $X$ if $\pi$ is flat and the diagram
\[
\xymatrix{
  X \ar[r]\ar[d]\ar@{}[dr]|-{\square} & X_A \ar[d]^{\pi} \\
  \Spec k \ar[r] & \Spec A
}
\]
is cartesian, that is, the special fiber $X_A \otimes_A k$ is isomorphic to $X$.  
\end{defin}

The data of all deformations of a certain geometric object can be conveniently described by a \emph{deformation functor}, that is, a covariant functor from the category $\Art_{W(k)}(k)$ to the category of sets satisfying certain technical conditions (\cite[Chapter 15]{hartshorne_deformation}).  In this paper, we shall be interested in the functors describing deformation of schemes and vector bundles potentially equipped with additional structure (e.g., equivariant structure).  The deformation functor of a scheme $X$ is defined by the association:

\[ 
  \Def_X \colon \Art_{W(k)}(k) \ra \cat{Sets}, \quad
  A \mapsto \genfrac{\{}{\}}{0pt}{} {\text{isom.\ classes of flat } X_A /\Spec A} {\text{with an identification } X_A \otimes_A k \isom X}.
\]

The set of all deformation functors forms a category with natural transformations as morphisms.  We say that a morphism of deformation functors $\cF \to \cG$ is \emph{smooth} if for every small extension $A' \to A$ the natural morphism $\cF(A') \to \cF(A) \times_{\cG(A)} \cG(A')$ induced by functoriality is surjective.  In particular, if $\cF \to \cG$ is smooth then for every $A \in \Art_{W(k)}(k)$ the map $\cF(A) \to \cG(A)$ is surjective too.  Note that this is consistent with standard notion of formal smoothness of morphism of schemes (see, e.g., \stacksproj{02H0}).   In the classical paper \cite[Proposition~3.10]{schlessinger} (see also \cite[Chapter 16]{hartshorne_deformation}), it is proven that for affine schemes with isolated singularities and projective schemes the above functor admits a hull, i.e., a smooth morphism from a functor $\Hom(R,-)$ for a complete $W(k)$-algebra $R$.   

\begin{remark}
In this paper, we shall often work with schemes admitting at most one lifting over every algebra $A \in \Art_{W(k)}(k)$.  We call such schemes \emph{infinitesimally rigid}.  An example of an infinitesimally rigid scheme is given by the projective space $\PP^n_{k}$, for every $n \geq 1$.
\end{remark}

Assume that $X$ admits a $W(k)$-lifting $\tilde X$.  Then for a given vector bundle $\cE$ on $X$, we consider a deformation functor of $\cE$ given by:
\[ 
  \Def_{\cE} \colon \Art_{W(k)}(k) \ra \cat{Sets}, \quad
  A \mapsto \genfrac{\{}{\}}{0pt}{} {\text{isom.\ classes of vector bundles } \cE_A \text{ on }\tilde X \otimes_{W(k)} A} {\text{together with an identification } \cE_{A|X} \isom \cE}.
\]
The functor depends on the choice of the lifting $\tilde X$ (e.g. for line bundles on K3 surfaces), but is clearly unambiguous for infinitesimally rigid schemes. 

In what follows we shall need a criterion for formal smoothness of a natural transformation of deformation functors.  It is based on two notions of tangent space and obstruction theory of the deformation functor $\cF$.  

\begin{defin}
The \emph{tangent} space $\Tan_{\cF}$ is defined by $\Tan_{\cF} = \cF(k[\varepsilon]/\varepsilon^2)$, and under the technical condition mentioned above -- satisfied in our context, admits a natural structure of a $k$-vector space.  The tangent space satisfies the following crucial property: for every small extension $A' \to A$ with kernel $I$ the morphism $\cF(A') \to \cF(A)$ is a pseudo-torsor under $I \otimes_k \Tan_{\cF}$.  
\end{defin}
\noindent For the explanation of the notion of obstruction theory, we refer to \cite[Section 3]{fantechi_manetti}.

\begin{lemma}[{\cite[Lemma 6.1]{fantechi_manetti}}]
\label{lem:smoothness_criterion_fantechi_manetti}
  Suppose $\left(\cF,(\Ob_\cF,\nu^\cF_e)\right)$ and $\left(\cG,(\Ob_\cG,\nu^\cG_e)\right)$ are deformation functors together with associated obstruction theories.  Let $\psi \colon \cF \to \cG$ be a morphism of functors admitting an obstruction map $\Ob_\psi \colon \Ob_\cF \to \Ob_\cG$. Then $\psi$ is smooth if the following conditions hold:
  \begin{enumerate}[i)]
    \item $\Tan_\psi \colon \Tan_\cF \to \Tan_\cG$ is surjective,
    \item $\Ob_\psi \colon \Ob_\cF \to \Ob_\cG$ is injective.
  \end{enumerate}
\end{lemma}

\begin{remark}
\label{remark:infinitesimally_rigid}
For a smooth projective scheme $X$ over $k$ the deformation functor described above admit obstruction theories satisfying (see \cite[Proposition 3.1.5, p. 248]{illusieI}):
\begin{enumerate}[i)]
  \item $\Tan_{\Def_X} \isom H^1(X,\cT_X)$ and $\Ob_{\Def_X} \isom H^2(X,\cT_X)$,
  \item $\Tan_{\Def_\cE} \isom H^1(X,\cEnd(\cE))$ and $\Ob_{\Def_\cE} \isom H^2(X,\cEnd(\cE))$.
\end{enumerate}
Consequently, the condition $H^1(X,\cT_X) = 0$ is sufficient and necessary for $X$ to be infinitesimally rigid.  In particular, by \cite[Th\'{e}oreme 2]{demazure} the Grassmanian $\Gr(k,N)$ is infinitesimally rigid, for every $k < N$.  Furthermore, the second set of conditions is true more generally for deformation of locally free modules on a ringed topos.  We shall use this property for sheaves on the quotient stack $[X/G]$.  Moreover, for every morphism $f \colon \cX \to \cY$ of ringed topoi flat over $W$ and a locally free sheaf $\cE$ on $Y = \cY \otimes k$, there exists a morphism of deformation functors $\Def_{\cE} \to \Def_{f^*E}$, induced by the pullback, whose tangent and obstruction maps are given by the natural maps $f^* \colon H^i(Y,\cEnd(\cE)) \to H^i(X,\cEnd(f^*\cE))$, for $i = 1,2$.
\end{remark}

\noindent We also need the following result whose proof is based on \cite[Proposition~2.2]{liedtke_satriano}.

\begin{prop}
\label{prop:defs_projective_bundles}
Assume that $X$ is an infinitesimally rigid scheme over $k$ such that $H^2(X,\cO_X) = 0$.  Then, for every vector bundle $\cE$, the natural morphism of deformation functors $\Def_{\cE} \to \Def_{\PP(\cE)}$ is smooth.
\end{prop}
\begin{proof}
Let $A' \to A$ be a small extension of rings in $\Art_{W(k)}(k)$.  We need to prove that the natural map
\[
\Def_{\cE}(A') \to \Def_{\cE}(A) \times_{\Def_{\PP(\cE)}(A)} \Def_{\PP(\cE)}(A')
\]
is surjective.  For this purpose, we take an $A'$-deformation $\tilde \PP$ of $\PP(\cE)$ such that $\tilde \PP_A \isom \PP(\cE_A)$, for some deformation $\cE_A$ of the vector bundle $\cE$ over a unique $A$-deformation $X_A$.  By the assumption $H^2(X,\cO_X) = 0$, we easily see that $H^2\left(\PP(\cE),\cO_{\PP(\cE)}\right) = 0$ and therefore all line bundles deform on $\tilde \PP_A$.  In particular, the tautological line bundle $\cO_{\tilde \PP_A}(1)$ deforms to a line bundle $\cO_{\tilde \PP}(1)$.  Using \cite[Proposition~2.2]{liedtke_satriano}, we see that the morphism $\pi \colon \PP(\cE_A) \to X_A$ also lifts to $\tilde \pi \colon \tilde \PP \to \tilde X$.  Consequently, the natural surjective evaluation map 
\[
\tilde\pi^*\tilde\pi_*\cO_{\tilde \PP}(1) \to \cO_{\tilde \PP}(1)
\] 
yields a morphism $\tilde \PP \to \PP(\tilde\pi_*\cO_{\tilde \PP}(1))$ of flat schemes over $A'$, which restricts to the isomorphism over $A$, and is therefore an isomorphism.  This finishes the proof.
\end{proof}


\subsection{Groups actions, quotient stacks and equivariant sheaves}
\label{sub:preliminaries_stacks_equivariant_sheaves}

We now give a short recollection of basic facts concerning group actions, classifying and quotient stacks, and equivariant sheaves.  A reader familiar with all these notions can freely skip this section and proceed to the main part of the paper.  Throughout this section $X$ is an $S$-scheme equipped with an action of a flat $S$-group scheme $ G$ (not necessarily smooth).  We begin by recalling the definition of an equivariant bundle on $X$.  In this section, by $m \colon G \times_S X \to X$ we denote the action map of $G$ on $X$, and by $p_X \colon G \times_S X \to X$ the projection onto the second factor.
 
\begin{defin}
A \emph{$G$-equivariant structure} on a quasi-coherent sheaves $\cE$ is an isomorphism $\sigma \colon p_X^*\cE \isom m^*\cE$ satisfying the standard cocycle condition (see \stacksproj{043S}).  A $G$-equivariant bundle is a vector bundle together with a choice of a $G$-equivariant structure.
\end{defin}

\noindent In turns out that equivariant sheaves can be conveniently described as objects on the \emph{quotient stack} $[X/G]$.  The stack $[X/G]$ is defined as the category fibered in groupoids over $\Sch^{\fppf}_S$ with:

\begin{itemize}
\item objects given by diagrams
\[
\xymatrix{
  P \ar[d]\ar[r] & X \\
  T,
}
\]
where $T$ is an $S$-scheme, the morphism $P \to T$ is a $G$-principal bundle and $P \to X$ is an equivariant map,
\item morphisms given by $G$-bundle maps
\[
\xymatrix{
  P \ar[d]\ar[r]\ar@{}[dr]|-{\square} & P'\ar[d] \\
  T \ar[r] & T'
}
\]
compatible with $P \to X$ and $P' \to X$.
\end{itemize}

\noindent In particular, for every $T \in \Sch^{\fppf}_{S}$ the morphisms $T \to [X/G]$ are in one-to-one correspondence with diagrams as above.  In the special case $X = S$, we denote $[S/G]$ by $BG$.  By above considerations, $BG$ is a moduli stack of principal $G$-bundles, that is, the set of morphisms $T \to BG$ is naturally bijective with the set of isomorphism classes of principal $G$-bundles on $T$.  By \stacksproj{044O} the stack $[X/G]$ is in fact a quotient of $X$ be the groupoid:
\[
\xymatrix{
 G \times_S X \ar@<-.5ex>[rr]_{p_X} \ar@<.5ex>[rr]^m & & X \ar[r] & [X/G],
}
\]
and therefore, using {\stacksproj{06WT}}, we obtain the following characterization of equivariant sheaves.

\begin{prop}
\label{prop:equivariant_and_stacky}
The category of quasi-coherent $G$-equivariant sheaves on $X$ is equivalent to the category of quasi-coherent sheaves on the quotient stack $[X/G]$.
\end{prop}

\begin{remark}
Motivated by the above result, throughout the paper we identify equivariant sheaves and associated sheaves on the quotient stack.  For an equivariant sheaf $\cE$ there is a natural notion of cohomology.  It is customary to refer to $H^i([X/G],\cE)$ as equivariant cohomology (more topologically oriented authors denote analogous groups by $H^i_G(X,\cE)$).   
\end{remark}

\begin{example}
If $G$ is a group scheme over a field $k$, then quasi-coherent sheaves on $[\Spec k/G]$ correspond to $G$-equivariant sheaves on $\Spec k$, that is, representations of $G$. 
\end{example}

\begin{cor}
\label{cor:equivariant_bundle_on_homogeneous_spaces}
Let $G$ be a flat group scheme over $S$, and let $P$ be a flat subgroup scheme.  Then the category of $G$-equivariant sheaves on the homogeneous space $G/P$ is equivalent to the category of $P$-representations. 
\end{cor}
\begin{proof}
By Proposition~\ref{prop:equivariant_and_stacky}, we need to describe the category of quasi-coherent sheaves on the quotient stack $[(G/P)/G]$.  Since the $P$-action on $G$ is free we see that the homogeneous space $G/P$ is isomorphic to the quotient stack $[G/P]$.  Consequently, the stack $[(G/P)/G]$ is isomorphic to the quotient of $G$ by the $G \times_S P$ action given by right multiplication by $G$ and left multiplication by $P$.  This quotient can be conducted in two steps: first dividing by $G$, and then by $P$.  This implies that 
\[
[(G/P)/G] \cong [G/G \times P] \cong [[G/G]/P] \cong BP_S, 
\]
and hence the proof is finished.
\end{proof}


\section{Deformations of equivariant sheaves}
\label{sec:equivariant_sheaves}

In this section, we analyze deformation theory of equivariant sheaves.  In particular, we introduce the functor of equivariant deformations of such sheaves and relate it to the standard functor of deformations ignoring the equivariant structure.  


\subsection{The Cartan--Leray spectral sequence}
\label{sub:equivariant_sheaves}

In order to compute the equivariant cohomology we shall use the following gadget coming from the realm of topology.  Every \emph{free} group action on a topological space gives rise to a Cartan--Leray spectral sequence relating cohomology groups of the space with the cohomology of the quotient.  Since the natural quotient map $X \to [X/G]$ is free it is natural to expect the existence of a similar sequence in the world of algebraic geometry.  Indeed, we have the following:

\begin{prop}
\label{prop:cartan_leray_spectral_seq}
Let $X/S$ be a scheme equipped with an action of a smooth groups scheme $G/S$.  Let $\cE$ be a $G$-equivariant sheaf (equivalently a sheaf on the quotient stack $[X/G]$) .  Then there exists a convergent $E_2$-page spectral sequence:
\[
E^{pq}_2 = H^p(G,H^q(X,\cE)) \Longrightarrow H^{p+q}([X/G],\cE).
\] 
\end{prop}
\begin{proof}
We apply the Leray spectral sequence (see \stacksproj{0734}) for the composition $[X/G] \to [\Spec(k)/G] \to \Spec k$.  The direct images via the first morphism can be computed using the $G$-equivariant Godement resolution.
\end{proof}

\begin{cor}
\label{cor:seven_term_Cartan}
Let $X/S$ be a scheme equipped with an action of a smooth groups scheme $G/S$.  Let $\cE$ be a $G$-equivariant sheaf (equivalently a sheaf on the quotient stack $[X/G]$).  Then there exists an exact sequence:
{\footnotesize
\begin{center}
\begin{tikzpicture}[descr/.style={fill=white,inner sep=1.5pt}]
        \matrix (m) [
            matrix of math nodes,
            row sep=1em,
            column sep=2.5em,
            text height=1.5ex, text depth=0.25ex
        ]
        { 0 & H^1(G,H^0(X,\cE)) & H^1([X/G],\cE) & H^1(X,\cE)^G \\
            & H^2(G,H^0(X,\cE)) & \Ker\left(H^2([X/G],\cE) \rightarrow H^2(X,\cE)\right) & H^1\left(G,H^1(X,\cE)\right), \\ };

        \path[overlay,->, font=\scriptsize,>=latex]
        (m-1-1) edge (m-1-2)
        (m-1-2) edge (m-1-3)
        (m-1-3) edge (m-1-4)
        (m-1-4) edge[out=355,in=175] (m-2-2) 
        (m-2-2) edge (m-2-3)
        (m-2-3) edge (m-2-4);
\end{tikzpicture}
\end{center}
}
\noindent where the maps $H^i([X/G],\cE) \ra H^i(X,\cE)$ in the sequence are the natural maps induced by the pullback.
\end{cor}
\begin{proof}  
For the existence of the sequence we apply the seven-term exact sequence associated with a spectral sequence.  The second claim follows from the construction of Cartan--Leray spectral sequence.
\end{proof}


\subsection{Deformation theory of equivariant sheaves}
\label{subsec:equiv_deform_sheaves}

If the scheme $\tilde X$ is equipped with an action over $W(k)$ of a group scheme $\tilde G$, and $\cE$ is a vector bundle on $X$ equivariant with respect to the action of $G = \tilde G \otimes k$ we may consider the functor of equivariant deformations of $\cE$: 

\[ 
  \Def^{G}_{\cE} \colon \Art_{W(k)}(k) \ra \cat{Sets}, \quad
  A \mapsto \left\{\begin{gathered} 
\text{isomorphism classes of vector bundles $\cE_A$ on $X_A$} \\
\text{equivariant with respect to $G_A = \tilde G \otimes_{W(k)} A$ } \\
\text{together with an equivariant identification $\cE_{A|X} \isom \cE$}
\end{gathered}\right\}.
\]

\noindent It turns out that the morphisms in the above spectral sequence describe the behaviour of the natural forgetful transformation $\Def^G_{\cE} \to \Def_{\cE}$.

\begin{thm}
\label{thm:def_equiv_vb}
Let $\tilde X/W(k)$ be a flat $\tilde G$-scheme and let $\cE$ be a $G$-equivariant vector bundle on $X$.  Then the natural forgetful transformation of deformation functors $\Def^G_{\cE} \to \Def_{\cE}$ is smooth if the following conditions are satisfied:
\begin{enumerate}[i)]
   \item $H^2\left(G,H^0(X,\cEnd(\cE))\right) = 0$,
   \item $H^1(X,\cEnd(\cE))$ is trivial as a $G$-module,
   \item $H^1(G,k) = 0$.  
 \end{enumerate} 
\end{thm}
\begin{proof}
Using Proposition~\ref{prop:equivariant_and_stacky}, we see that the transformation $\Def^G_{\cE} \to \Def_{\cE}$ is in fact induced by the system of pullback morphisms $\pi_A^*$ coming from the projections $\pi_A \colon X_A \to [X_A/G_A]$, for all $A \in \Art_{W(k)}(k)$.  By Remark~\ref{remark:infinitesimally_rigid}, the tangent and obstruction maps of the transformation can be identified with the pullbacks $H^i\left([X/G],\cEnd(\cE)\right) \to H^i\left(X,\cEnd(\cE)\right)$, for $i=1,2$ respectively.  By Corollary~\ref{cor:seven_term_Cartan}, those fit into an exact sequence
{\scriptsize 
\begin{center}
\begin{tikzpicture}[descr/.style={fill=white,inner sep=1.5pt}]
        \matrix (m) [
            matrix of math nodes,
            row sep=1em,
            column sep=2.5em,
            text height=1.5ex, text depth=0.25ex
        ]
        { 0 & H^1\left(G,H^0(X,\cEnd(\cE))\right) & H^1([X/G],\cEnd(\cE)) & H^1(X,\cEnd(\cE))^G \\
            & H^2\left(G,H^0(X,\cEnd(\cE))\right) & \Ker\left(H^2([X/G],\cEnd(\cE)) \rightarrow H^2(X,\cEnd(\cE))\right) & H^1\left(G,H^1(X,\cEnd(\cE))\right), \\ };

        \path[overlay,->, font=\scriptsize,>=latex]
        (m-1-1) edge (m-1-2)
        (m-1-2) edge (m-1-3)
        (m-1-3) edge (m-1-4)
        (m-1-4) edge[out=355,in=175] (m-2-2) 
        (m-2-2) edge (m-2-3)
        (m-2-3) edge (m-2-4);
\end{tikzpicture}
\end{center}
}

\noindent To conclude using Lemma~\ref{lem:smoothness_criterion_fantechi_manetti}, it is sufficient to prove that $H^1(X,\cEnd(\cE))$ is a trivial $G$-module and $H^1\left(G,H^1(X,\cEnd(\cE))\right) = H^2\left(G,H^0(X,\cEnd(\cE))\right) = 0$.  These statements follow directly from the assumptions.
\end{proof}

\begin{remark}\label{rem:obs_descr}
The above conditions have natural interpretations.  For example, the element in $H^2(G,H^0(X,\cEnd(\cE)))$ is an obstruction for the extension of the deformation of the equivariant structure.  In the case of $G$ a finite group we can describe it as follows.  The equvariant structure over $A \in \Art_{W(k)}(k)$ is given by a $1$-cocycle $\sigma \colon G \to \Aut(\cE_A)$.  An extension over $A' \in \Art_{W(k)}(k)$ of such structure is a lifting of $\sigma$ which satisfy the cocyle condition.  An obstruction for this is a $2$-cocycle in the tangent space (identified with $H^0(X,\cEnd(\cE))$) of the automorphism functor, that is, an element of $H^2(G,H^0(X,\cEnd(\cE)))$, as desired.
\end{remark}

\begin{remark}
We see that the conditions are necessary by a simple example of $\SL_2(k)$-equivariant bundle on $X = \PP^1_k$ equipped with the natural action.  Let $\cE = \cO \oplus \cO(2n)$, for $n \geq 2$.  The space $H^1(X,\cEnd(\cE))$ is isomorphic to the $\SL_2(k)$-representation 
\[
H^0(\PP^1,\cO(2n-2)) = k[x_0,x_1]_{2n-2},
\]
with no fixed points, and therefore no non-trivial deformation of $\cE$ admits an extension of the equivariant structure. Interestingly, the situation is different for the vector bundle $\cO \oplus \cO(2)$. 
\end{remark}


\subsection{Deformation theory of equivariant sheaves on homogeneous spaces}
\label{subsec:equiv_deform_sheaves_homogeneous_spaces}

In this section, we briefly analyze the functor of equivariant deformations in the special case of homogeneous spaces defined over $W(k)$.  Suppose $\tilde G$ is a smooth affine group scheme over $W(k)$, and $\tilde P \subset \tilde G$ is a parabolic subgroup scheme.  Let $\tilde X = \tilde G/\tilde P$ be the associated homogeneous space over $W(k)$, and let $\cF$ be a rank $k$ vector bundle on $X = \tilde X \otimes k$ equivariant with respect to the action of $G = \tilde G \otimes k$.  

By Corollary~\ref{cor:equivariant_bundle_on_homogeneous_spaces} we know that for every $A \in \Art_{W(k)}(k)$ the category of $G_A$-equivariant sheaves on $X_A = G_A/P_A$ is equivalent to the category of $P_A$-representations.  Representations of rank $k$ are in one-to-one correspondence with orbits of the conjugacy action of $\GL_{k,A}$ on the set of homomorphisms $\Hom(P_A,\GL_{k,A})$.  Taking $f \colon P \to \GL_k$ to be one of the homomorphisms corresponding to the sheaf $\cF$ on $G/P$, we obtain a natural transformation of deformation functors $\Def_f \ra \Def^{G}_{\cF}$ (cf. Remark~\ref{remark:deformation_functor_of_homomorphism}).  

\begin{lemma}
\label{lem:g_equivariant_deform_morphism}
For every $G$-equivariant sheaf $\cF$ on $G/P$ and every homomorphism $f \colon P \to \GL_k$ giving rise to $\cF$, the natural transformation $\Def_f \ra \Def^{G}_{\cF}$ is smooth.  In particular, every equivariant deformation of $\cF$ yields a deformation of the associated homomorphism.
\end{lemma}
\begin{proof}
Simple application of formal smoothness of $G$.
\end{proof}


\section{Reductive groups}
\label{sec:preliminaries_reductive_groups}

In this section, we recall standard result concerning reductive groups over arbitrary base.  As a reference we use the classical account from \cite{SGA3III}.

\begin{defin}[{\cite[Expos\'{e} XIX, D\'efinition 2.7]{SGA3III}}]
We say that an affine group scheme $G$ defined over on algebraically closed field is \emph{reductive} if its unipotent radical, i.e., maximal connected normal unipotent group is trivial.  Over a general base $S$, a group scheme $\cG \to S$ is reductive if it is affine and flat, and for every geometric point $\bar s$ the fibre $\cG_{\bar s}$ is reductive.
\end{defin}

\begin{example}
The examples most relevant in this paper are the standard linear groups $\SL_{N/\ZZ}$, $\Sp_{2n/\ZZ}$, $\SO_{N/\ZZ}$ and $\GL_{N/\ZZ}$ defined over the integers.  
\end{example}

\begin{defin}
A group scheme $G$ over a field is \emph{linearly reductive} if its category of representations is semi-simple, that is, for every representation $V$ and a subrepresentation $W \subset V$ there exists an invariant complement.
\end{defin}

\begin{example}
An example of linearly reductive groups is given by diagonalizable groups, that is, subgroups of a group of diagonal matrices.  For the definitions and a proof, see \cite[Expos\'e VIII]{SGA3III} or \cite[3.2.3 Theorem]{springer}.
\end{example}


\subsection{Root data}
\label{ss:reductive_groups_root_data}

We now recall the standard notion of root data.  The definition over an algebraically closed field is included in \cite[Section 7.4]{springer}.  For more general results, necessary in our context, we refer to \cite[Expos\'e XIX, XII]{SGA3III}. 

Let $S$ be a connected scheme, and let $G/S$ be a reductive group scheme.  Take $T \subset G$ to be a maximal torus over $S$, that is, a subtorus of $G$ such that for every geometric point $\overline{s} \to S$ the group $T_{\overline{s}} \subset G_{\overline{s}}$ is a maximal torus in the standard sense (see \cite[Expos\'e XII, D\'efinition 1.3]{SGA3III}).  Such torus exists \'etale locally on the base $S$.  The torus $T$ acts via conjugation on the Lie algebra $\Lie{g} = e^*\Omega^1_{G/S}$, where $e \colon S \to G$ is the identity section of $G$, decomposing it into a finite sum
\[
\Lie{g} = \Lie{t} \oplus \bigoplus_{\alpha \neq 0} \Lie{g}_{\alpha}
\] 
of eigenspaces $\Lie{g}_\alpha$ for characters $\alpha \in \Hom_{\rm grp/S}(T,\GG_{m/S})$, which are locally free sheaves on $S$.  The set of characters $\alpha$ appearing in the above decomposition is called the set of \emph{roots} of $G/S$ with respect to $T$.  We denote this set by $\cR(G,T)$.  It turns out that the functor 
\[
\cR \colon \Sch_{S} \to \cat{Sets} \quad S' \mapsto \cR(G_{S'},T_{S'})
\]
is representable by a closed and open subscheme $R$ of the scheme $\underline{\Hom}_{\rm grp/S}(T,\GG_{m/S})$.
Furthermore, by \cite[Expos\'e XXII, Th\'eorème 1.1]{SGA3III}, there exists a functorial choice of a set $\cR^\vee(G,T) \subset \Hom_{\rm grp/S}(\GG_{m/S},T)$, called the set of \emph{coroots}, along with a natural bijection 
\[
\cR(G,T) \to \cR^\vee(G,T) \quad \alpha \mapsto \alpha^\vee 
\] 
such that the pairing induced by composition
\[
\scalar{-,-} \colon \underline{\Hom}_{\rm grp/S}(T,\GG_{m/S}) \times \underline{\Hom}_{\rm grp/S}(\GG_{m/S},T) \to \underline{\Hom}_{\rm grp/S}(\GG_{m/S},\GG_{m/S}) = \ZZ_S.
\]
satisfies the condition $\scalar{\alpha,\alpha^\vee} = 2$.  The functor \[
\cR^\vee \colon \Sch_{S} \to \cat{Sets} \quad S' \mapsto \cR^\vee(G_{S'},T_{S'})
\]
is representable by a closed and open subscheme $R^\vee \subset \underline{\Hom}_{\rm grp/S}(\GG_{m/S},T)$, and together with the natural association $\alpha \to \alpha^\vee$ satisfies number of other conditions which make a tuple
\[
(X(T) = \underline{\Hom}_{\rm grp/S}(T,\GG_{m/S}),Y(T) = \underline{\Hom}_{\rm grp/S}(\GG_{m/S},T),R,R^\vee),
\]
a \emph{root data} for a pair $(G,T)$ over $S$.  We refer to \cite[Expos\'e XXII, D\'efinition 1.9]{SGA3III} for a precise definition, not necessary in our context.  

\subsection{Rigidification}
\label{ss:reductive_groups_rigidification}

\'Etale locally on the base $S$, the torus $T$ is trivial, i.e., there exists a torsion free $\ZZ$-module $M$ such that $T$ is isomorphic to the tori $D_S(M) = \Hom_{\ZZ}(M, \GG_{m/S})$.  In this situation, the root (resp. coroot) subscheme is isomorphic to the total space of a constant subsheaf of $M_S$ (resp. $M_S^\vee$) associated with a subset $R \subset M$ (resp. $R^\vee \subset M^\vee$).  Furthermore, also \'etale locally, the root eigenspaces $\Lie{g}_{\alpha}$ are in fact free $\cO_S$-modules.  If those conditions are satisfied on $S$, we say that the group scheme $G/S$ is \emph{rigidified} (fr. \emph{d\'eploy\'e}) with respect to $T$ (see \cite[Expos\'e XXII, D\'efinition 1.13]{SGA3III}).  There is a natural notion of a morphism of rigidified groups. 

\begin{defin}[{\cite[Expos\'e XXII \S 4]{SGA3III}}]
\label{def:rigidified_morphism}
Let $(G,T,M,R)$ and $(G',T',M',R')$ be two rigidified reductive groups over $S$ together with their root data.  A morphism of $f \colon G \to G'$ is \emph{rigidified} if the homomorphism $f_{|T}$ factors through a homomorphism of tori $D_S(h) \colon T \to T'$ induced by a morphism of lattices $h \colon M' \to M$ such that
\begin{center}\begin{minipage}{13cm}
there exists a bijection $d \colon R \to R'$ and a function $q \colon R \to \NN$ such that the association $x \mapsto x^{q(\alpha)}$ defines an endomorphism of $\GG_{a/S}$ and the relations
\[
h(d(\alpha)) = q(\alpha)\cdot \alpha \quad \quad h^t(\alpha^\vee) = q(\alpha)\cdot d(\alpha)^\vee
\]
are satisfied.
\end{minipage}\end{center}

\begin{remark}
\label{remark:function_q_and_char}
It is important to remark that the association $x \to x^q$ induces an endomorphism of $\GG_{a/S}$ if and only if $q = p^n$ for $n \geq 0$ and a prime number $p$ satisfies $p = 0 \in \cO_S$.  In particular, for schemes $S$ where no prime number is equal to zero the function $q \colon R \to \NN$ is necesarilly identically equal to one.
\end{remark}
\end{defin}

\noindent We use the above notion for isogenies of reductive groups.  We say that a morphism $G \to G'$ of reductive groups over $S$ is an \emph{isogeny} if it is faithfully flat and finite (see \cite[Expos\'e XXII, D\'efinition 4.9]{SGA3III}).  

\begin{example}
The natural quotient map $\SL_2 \to \PGL_2$ is a rigidified isogeny with kernel $\mu_2$.  The associated morphism of lattices is the multiplication by two and the function $q$ is identically one.
\end{example}

\begin{example}
Every reductive group over a field $k$ of characteristic $p>0$ admits a Frobenius endomorphism.  The associated morphism of lattices is the multiplication by $p$ and the function $q$ is identically $p$.  Note that the induced morphism of Lie algebras is zero.
\end{example}

\begin{prop}[{\cite[Expos\'e XXII, Corollaire 4.2.13]{SGA3III}}]
\label{prop:morphism_rigidified}
Every isogeny $G \to G'$ of reductive groups over $S$ is rigidified \'etale locally on $S$.  In particular, every isogeny of reductive groups is rigidified over a spectrum of an Artinian algebra.
\end{prop}


\section{Deformations of the Frobenius morphism of group schemes}
\label{sec:frobenius_group_schemes}

After the above preparation, we are ready to present an important part of the paper pertaining to the deformations of the Frobenius homomorphism.  More precisely, in this section, we introduce the deformation functor of the Frobenius homomorphism of a group scheme defined over the ring of Witt vectors of a perfect field $k$ of characteristic $p>0$.  Moreover, based on the previous section, we prove that for a reductive group the functor is restricted to rings where $p = 0$.


\subsection{General definitions}
\label{sub:def_frob_general_defs}

We start with a definition.  Let $\tilde G$ be a flat group scheme over $W(k)$, and let $\tilde{G'} = \tilde G \otimes_{W(k),\sigma} W(k)$ be the base change of $\tilde G$ along the Frobenius morphism $\sigma \colon W(k) \to W(k)$.  Moreover, for every $A \in \Art_{W(k)}(k)$, let $G_A = \tilde G \otimes_{W(k)} A$.  We say that the a morphism $F_A \colon G_A \to G'_A$ over $A$ is a \emph{lifting of the Frobenius homomorphism} if it is a homomorphism of group schemes over $A$ and its restriction $F_A \otimes k$ is equal to the relative Frobenius morphism.  The data of liftings of Frobenius homomorphism, for all $A \in \Art_{W(k)}(k)$, can be handily packed in the following deformation functor 
\[ 
  \Def_{F_{\tilde G}} \colon \Art_{W(k)}(k) \ra \cat{Sets}, \quad
  A \mapsto \left\{ \text{liftings of the Frobenius homomorphism over $A$} \right\}.
\]

%

\begin{example}
\label{example:frobenius_multiplicative_type}
Let $M$ be a finitely generated $\ZZ$-module.  For every scheme $S$, the diagonalizable group $D_M(S) = \Hom_{\ZZ}(M,\GG_{m/S})$ admits a multiplication by $p$ map induced by the multiplication by $p$ module endomorphism of $M$.  For $S = \Spec W(k)$ this morphism is a lifting of the Frobenius homomorphism of $D_m(k) = Hom_{\ZZ}(M,\GG_{m/k})$.  Since the scheme of endomorphism of a torus is discrete the deformations of Frobenius are trivial, and therefore $\Def_{F_{D_M}}$ is isomorphic to formal spectrum ${\rm Spf}(W)$. 
\end{example}

In this paper, we shall be mostly interested in the above functor considered for a deformation $\tilde G$ of a reductive group scheme $G$.  In this context, we emphasize that, by the existence and rigidity statements for redctive groups (see \cite[Expos\'e XXII, Corollaire 5.1]{SGA3III}) the definition of the deformation functor of the Frobenius morphism only depends on the group scheme $G$, and therefore we may consider a non-ambiguous deformation functor $\Def_{F_G}$.  We now apply the results of the previous section to describe this functor.

\begin{thm}\label{thm:defos_frob_red_gps}
Let $G/W$ be a reductive group scheme over $W$.  Then the functor $\Def_{F_G}$ satisfies the property $\Def_{F_G}(A) = \emptyset$, for every Artinian $W$-algebra $A$ such that $p \neq 0$ in $A$.
\end{thm}
\begin{proof}
This is a direct corollary of Proposition~\ref{prop:morphism_rigidified}.  More precisely, we take a homomorphism $F_A \colon G_A \to G'_A$ over $A$ lifting the relative Frobenius.  It is clearly rigidified (see Definition~\ref{def:rigidified_morphism}).  Morever, the data of a ridification is discrete and therefore uniquely defined by the Frobenius homomorphism at the special fibre.  This implies that the function $q$ in the definiton of the rigidification is identically equal to $p$.  By Remark~\ref{remark:function_q_and_char} we consequently see that $p = 0$ in $A$, which finishes the proof.
\end{proof}

\begin{remark}
On the contrary, for an abelian variety $E/k$ the behaviour of deformation functors $\Def_{F_{\tilde E}}$ considerably depend on the choice of the lifting $\tilde E$.  For instance, taking $E$ to be an ordinary abelian variety \maciek{express some arithmetic condition, extend the remark with the content of de Jong complex multiplication}, by Serre--Tate theory there exists a \emph{unique} lifting $\tilde E/W(k)$ such that $F \colon E \to E$ lifts formally to $W(k)$.
\end{remark}

\begin{remark}
\label{remark:deformation_functor_of_homomorphism}
More generally, given two group schemes $\tilde G$ and $\tilde H$ over $W(k)$ together with a morphism $f \colon G \to H$, we may consider the deformation functor $\Def_{f}$ defined by the formula 
\[ 
  \Def^{\tilde G,\tilde H}_{f} \colon \Art_{W(k)}(k) \ra \cat{Sets}, \quad
  A \mapsto \left\{ \text{homomorphisms $f_A \colon G_A \to H_A$ restricting to $f$} \right\}.
\]
The Frobenius deformation functor is a particular instance $\Def^{\tilde G,\tilde G'}_{F_G}$ of this definition.
\end{remark}


\section{Frobenius twists of tautological bundles on flag varieties}
\label{sec:applications_totaro}

In this section we apply the results developed above to analyze the deformation theory of Frobenius twists of tautological bundles on Grassmanians.  As a corollary we obtain a description of deformation theory of recent Totaro's examples.  In this section, by $k$ we denote an algebraically closed field.


\subsection{Some results in representation theory}
\label{sub:representation_theory}

We begin with some preliminary results on representation theory of reductive groups.  We refer to \cite[Chapter II]{jantzen} for the detailed account on the general theory and the comprehensive presentation of the characteristic $p>0$ phenomena.  

We now extended the results on root data described in section \S\ref{sec:preliminaries_reductive_groups} with some basics of representation theory.  Let $T$ be a torus, that is, a group isomorphic to a product of finitely many copies of $\GG_{m}$ (note that we work over an algebraically closed field, and hence all tori are split).  We set $X(T) = \Hom(T,\Gm)$ to be the group of characters and $Y(T) = \Hom(\Gm,T)$ to be the group of cocharacters.  Now, we take $G$ to be a reductive group over $k$.  We fix a maximal torus $T \isom \Gm^n$ in $G$, and a Borel subgroup $B$ containing $T$, i.e., a maximal closed connected solvable subgroup.  Given this data, as in \S\ref{ss:reductive_groups_root_data}, we consider the adjoint action of $T$ on the Lie algebra $\Lie{g}$.  The action is diagonalizable, and hence yields a set $R \subset X(T)$ of non-zero weights, which is called the \emph{roots} of $G$.  The subset $R^+$ of roots appearing in the decomposition of $T$-action on $\Lie{b} \subset \Lie{g}$ is called the \emph{positive roots}.  There exists a unique basis $S \subset R^+$ of $X(T)\otimes \RR$ such that any other element of $R^+$ can be written as a sum of positive multiples of elements of $S$.  We refer to $S$ as \emph{simple roots}.

\begin{example}
\label{example:root_data_gl_N}
We now give a few details concerning the root data of the linear group $\GL_n$.  We choose a maximal torus $T \isom \Gm^n \subset \GL_n$ to be the group of diagonal invertible matrices.  The group $X(T)$ is a lattice with a basis given by characters $\{\ell_i\}_{i=1\ldots n} \subset X(T)$ defined by the formula $\ell_i\left(\mathrm{diag}(t_1,\ldots,t_n)\right) = t_i$.  Taking $B \subset \GL_n$ to be a Borel subgroup of upper triangular matrices, the root data is as follows:

\[
\begin{gathered}
\text{roots $R$: $\ell_i - \ell_j$ for $1 \leq i \neq j \leq n$,} \\ 
\text{positive roots $R^+$: $\ell_i - \ell_j$ for $1 \leq i < j \leq n$,} \\
\text{simple roots $S$: $\ell_i - \ell_{i+1}$ for $1 \leq i < n$,} \\
\text{half sum of positive roots: $\rho = \sum_{i=1}^n (n+1 - 2i)\ell_i$.}
\end{gathered}
\]

\noindent Under the identification of the root space $X(T)_{\RR} = X(T) \otimes \RR$ and the coroot space $Y(T)_{\RR} = Y(T) \otimes \RR$ induced by the scalar product $\scalar{-,-}$ on $X(T)_\RR$, determined by an orthonormal basis $\{\ell_i\}_{i=1\ldots n}$, we further have:

\[
\begin{gathered}
\text{coroots $R^\vee$: $(\ell_i - \ell_j)^\vee = \ell_i - \ell_j$.} \\
\end{gathered}
\]

\noindent The reflections and corresponding dot actions are explicitly given by the relations:
\[
\begin{aligned}
s_{\ell_i - \ell_{i+1}}(\ell_i) = \ell_{i+1} \quad\quad s_{\ell_i - \ell_{i+1}}(\ell_{i+1}) = \ell_i, \\
s_{\ell_i - \ell_{i+1}}(\ell_k) = \ell_k \text{, for $k \not\in \{i,i+1\}$} \\ 
s_{\ell_i - \ell_{i+1}} \cdot \lambda = s_{\ell_i - \ell_{i+1}}(\lambda) - (\ell_i - \ell_{i+1}).
\end{aligned}
\]
\end{example}

\subsection{Cohomology groups of equivariant sheaves}
\label{ss:cohomology_representation_thy}

For every character $\lambda \in X(T)$, we denote by $\cL_{\lambda}$ the line bundle on $G/B$ given by $G \times_B k_{\lambda}$, where $k_\lambda$ is a $B$ representation induced by the homomorphism $B \to T \to \GG_m$.  In order to apply Theorem~\ref{thm:def_equiv_vb} we need to control the behaviour of certain cohomology groups of equivariant bundles.  Unfortunately, the usual tool applied in this context --- Borel--Weil--Bott theorem --- fails in characteristic $p>0$.  As a substitute we shall use the following results due to Kempf and Andersen.  Before stating the result, we recall the definition of the dot action $\lambda \mapsto s_{\alpha}\cdot \lambda$.  For every simple root $\alpha \in S$, it is defined by the formula $s_{\alpha} \cdot \lambda = s_\alpha(\lambda + \rho) - \rho$, where $s_\alpha$ is the reflection associated with the root $\alpha$ and $\rho$ is the half sum of all positive roots.

\begin{prop}[{\cite[Proposition~II.4.5]{jantzen}}]
\label{prop:jantzen_kempf_vanishing}
Let $\lambda \in X(T)$ be a weight.  Then $H^0(G/B,\cL_{\lambda}) \neq 0$ if and only if $\lambda$ is dominant.  Moreover, if $\lambda$ is dominant then $H^i(G/B,\cL_{\lambda}) = 0$, for $i \geq 1$.
\end{prop}

\begin{prop}[{\cite[Proposition~II.5.15]{jantzen}}]
\label{prop:jantzen_bwb}
Let $\alpha$ be a simple root and $\lambda \in X(T)$ with $\scalar{\lambda,\alpha^\vee} > 0$.
\begin{enumerate}[a)]
  \item Suppose $\scalar{\lambda,\alpha^\vee} = ap^n - 1$, for $0 < a < p$ and $n \geq 0$. Then 
  \[
H^1(G/B,\cL_{s_\alpha \cdot \lambda}) \neq 0 \quad \iff \quad \lambda \text{ is dominant}.
\]
  \item Suppose $\scalar{\lambda,\alpha^\vee} = \sum_{j = 0}^n a_j p^j$ with $0 \leqslant a_j < p$ for all $j$ and $a_n \neq 0$.  Assume that there exists $j < n$ such that $a_j < p-1$.  Then
\[
H^1(G/B,\cL_{s_\alpha \cdot \lambda}) \neq 0 \quad \iff \quad s_\alpha \cdot \lambda + a_n p^n \alpha \text{ is dominant}.
\]
If the above condition holds and $\lambda$ is dominant, then $\lambda$ is the largest weight of $H^1(G/B,\cL_{s_\alpha \cdot \lambda})$.  Otherwise, let $m$ be minimal with $a_m < p-1$ and let $m' \geq m$ be minimal for the condition $\mu = s_\alpha \cdot \lambda + \sum_{j = m'}^n a_j p^j \alpha$ is dominant.  Then $\mu$ is the largest weight of $H^1(G/B,\cL_{s_\alpha \cdot \lambda})$. 
\end{enumerate}

\end{prop}

\begin{prop}[{\cite[Proposition~II.4.13]{jantzen}}]
\label{prop:jantzen_cohomology_trivial}
Let $G$ be a reductive group over a field $k$.  Then the cohomology groups $H^i(G,k)$ of the trivial $G$-module vanish for $i \geq 1$.
\end{prop}


\subsection{Deformation theory of Frobenius twists}
\label{sub:defo_thy_frob_twists}

We now combine the results of Theorem~\ref{thm:def_equiv_vb} and Theorem~\ref{thm:defos_frob_red_gps} to describe deformation theory of Frobenius twists of tautological vector bundles on Grassmanian varieties.  In this section, we interpret $\Gr(d,N)$ as the homogeneous space of the group $\GL_N$ parameterizing $d$-dimensional linear subspaces of a fixed $N$-dimensional vector space $V$.  The tautological bundle is the natural bundle with fiber isomorphic to $W$ for a point $[W \subset V] \in \Gr(d,N)$.  In what follows we shall freely use the description of root data of $\GL_N$ given in Example~\ref{example:root_data_gl_N}.

\begin{thm}
\label{thm:defos_flag_varieties}
Let $\cE$ be the tautological vector bundle on the Grassmanian variety $\Gr(d,N)$, for $2 \leq d \leq N - 2$.  Then $\PP(F^*\cE)$ does not lift to any ring where $p \neq 0$.
\end{thm}
\begin{proof}

The proof depends on the description of the $\GL_N$-equivariant structure of the bundle $F^*\cE$.  As in \S\ref{sub:representation_theory}, let $ B$ be the Borel subgroup of upper triangular matrices, and let $P$ be the parabolic subgroup of block matrices such that $\Gr(d,N) \cong \GL_N/P$.  Using Corollary~\ref{cor:equivariant_bundle_on_homogeneous_spaces}, we immediately see that $\cE$ is a vector bundle corresponding the representation $\pi \colon P \to \GL_k$ given by the projection onto the $d \times d$ block isomorphic to $\GL_d$.  Consequently, the vector bundle $F^*\cE$ is associated with the composition of the projection $\pi$ and the Frobenius homomorphism of $\GL_d$, which we~denote~by~$f$. 

We infer the claim of the proposition from the existence of the following diagram of \emph{smooth} morphisms of deformation functors
\[
\Def_{\PP(F^*\cE)} \longleftarrow \Def_{F^*\cE} \longleftarrow \Def^{\GL_N}_{F^*\cE} \longleftarrow \Def_{f} 
\longrightarrow \Def_{F_{\GL_d}}
\]
where the equivariant deformations of $F^*\cE$ are considered with respect to the natural lifting $\tilde P \subset \GL_{N,W(k)}$ of $P$ (cf. \S\ref{subsec:equiv_deform_sheaves} and \S\ref{subsec:equiv_deform_sheaves_homogeneous_spaces}).  We begin the proof of existence of the diagram by observing that, using Remark~\ref{remark:infinitesimally_rigid}, the Grassmanian is infinitesimally rigid and therefore there exists a morphism of deformation functors $\Def_{F^*\cE} \to \Def_{\PP(F^*\cE)}$, which is smooth by Proposition~\ref{prop:defs_projective_bundles}.  

We now prove that the forgetful natural transformation $\Def^{\GL_N}_{F^*\cE} \ra \Def_{F^*\cE}$ associated with the $\GL_N$-equivariant structure on $F^*\cE$, as described in \S\ref{sub:equivariant_sheaves}, is smooth.  We aim at applying Theorem~\ref{thm:def_equiv_vb}.  For this purpose, we first see that $\GL_N$ is reductive and therefore $H^1(\GL_N,k) = 0$ by Proposition~\ref{prop:jantzen_cohomology_trivial}, which yields condition iii) of Theorem~\ref{thm:def_equiv_vb}.  In order to prove that other assumptions are satisfied we first provide a simple technical lemma.

\begin{lemma}
\label{lem:technical_proof_totaro}
Let $G$ be a reductive group, and let $B$ be the Borel subgroup.  Suppose $\cF$ is an equivariant vector bundle on $G/B$ which is filtered by equivariant line bundles $\{\cL_i\}_{i=1\ldots n}$ such that $H^1(G/B,\cL_i)$ is either a trivial $G$-representation or zero.  Then $H^1(G/B,\cF)$ is either a trivial representation or zero. 
\end{lemma}
\begin{proof}
We reason by induction with respect to $n$.  For $n = 1$ the claim is clear.  For $n \geq 2$, we consider the long exact cohomology sequence associated with the short exact sequence
\[
0 \ra \cL_1 \ra \cF \ra \cF/\cL_1 \ra 0.
\]
We obtain a sequence
\[
\cdots \ra H^1(G/B,\cL_1) \ra H^1(G/B,\cF) \ra H^1(G/B,\cF) \ra \cdots.
\]
By induction hypothesis both right and left terms are either trivial or zero, and thus the middle term likewise (there are no non-trivial extensions by Kempf vanishing).
\end{proof}

We now proceed with the proof of the proposition.  We observe that the set of weights of the representation associated with $F^*\cE$ is equal to $\{p\ell_1,\ldots,p\ell_d\}$.  Consequently, the set of weights of the endomorphism representation is equal to $\{p(\ell_{i} - \ell_{j})\}_{1 \leq i,j \leq d} $.  The natural projection $\pi \colon \GL_N/B \to \GL_N/P$ satisfies the condition $R\pi_*\cO_{\GL_N/B} = \cO_{\GL_N/P}$, and hence by standard Leray spectral sequence argument we may compute cohomology after pulling back to $\GL_N/B$.  Now, we observe that $\pi^*\cEnd(F^*\cE)$ is filtered by line bundles associated with weights $p(\ell_i - \ell_j)$, for $1 \leq i,j \leq d$.  Since $p(\ell_i - \ell_{j})$ is not dominant for $i \neq j$, by Kempf vanishing (see Proposition~\ref{prop:jantzen_kempf_vanishing}) the group 
\[
H^0(\GL_N/P,\cEnd(F^*\cE)) = H^0(\GL_N/B,\pi^*\cEnd(F^*\cE))
\] 
is filtered by trivial $G$-representations, and consequently $H^2\left(\GL_N,H^0(\GL_N/P,\cEnd(F^*\cE))\right)$ is zero by Proposition~\ref{prop:jantzen_cohomology_trivial}.  This yields condition i) of Theorem~\ref{thm:def_equiv_vb}.  To finish the proof, we need to show that $H^1(\GL_N/P,\cEnd(F^*\cE)) = H^1\left(\GL_N/B,\pi^*\cEnd(F^*\cE)\right)$ is a trivial $\GL_N$-module.  To this end, we use the filtration of $\pi^*\cEnd(F^*\cE)$ with quotient $\cL_{p(\ell_i - \ell_j)}$ again.  In order to apply Lemma~\ref{lem:technical_proof_totaro}, we distinguish four cases. 
\begin{enumerate}
\item For $i = j$, the $1$-st cohomology group of the suitable line bundle clearly vanishes by Kempf vanishing.  
\item For $i < j$, we observe that 
\[
s_{\ell_{j} - \ell_{j+1}} \cdot (p(\ell_{i} - \ell_j)) = p(\ell_{i} - \ell_{j+1}) - (\ell_{j} - \ell_{j+1}),
\]
and hence for $\lambda = s_{\ell_{j} - \ell_{j+1}} \cdot (p(\ell_{i} - \ell_{j}))$ we have

\[
\scalar{\lambda,(\ell_j - \ell_{j+1})^\vee} = \scalar{p(\ell_{i} - \ell_{j+1}) - (\ell_j - \ell_{j+1}),\ell_j - \ell_{j+1}} = p-2.
\]
By checking scalar product with $\ell_{j+1} - \ell_{j+2}$, which is allowable because $j \leq N-2$, we see that $\lambda$ is not dominant.  This allows us to apply Proposition~\ref{prop:jantzen_bwb} a) to see that $H^1(G/B,\cL_{p(\ell_i - \ell_j)}) = 0$.
\item For $i > j+1$, we reason analogously using a simple root $\ell_{i-1} - \ell_{i}$ instead of $\ell_{j} - \ell_{j+1}$.  More precisely we see that
\[
\scalar{s_{\ell_{i-1} - \ell_{i}} \cdot p(\ell_{i} - \ell_{j}),\ell_{i-1} - \ell_{i}} = \scalar{p(\ell_{i-1} - \ell_{j}) - (\ell_{i-1} - \ell_{i}),\ell_{i-1} - \ell_{i}} = p-2,
\]
and therefore the same arguments as in (2) can be applied.  
\item Finally, for $i = j+1$, using a simple root $\ell_{i-1} - \ell_i$ and taking 
\[
\lambda = s_{\ell_{i-1} - \ell_i} \cdot p(\ell_{i} - \ell_{i-1}) = (p-1)(\ell_{i-1} - \ell_{i}),
\]
we compute as above to observe that
\[
\scalar{\lambda,(\ell_{i-1} - \ell_{i})^\vee} = \scalar{(p-1)(\ell_{i-1} - \ell_i),\ell_{i-1} - \ell_{i}} = 2p-2 = p + (p-2).
\] 
This allows us to apply Proposition~\ref{prop:jantzen_bwb} b) to see that $H^1(\GL_N/B,\cL_{p(\ell_i - \ell_{i-1})})$ is in fact non-zero of highest weight zero, and therefore trivial. 
\end{enumerate}
Summing up all the cases and applying Lemma~\ref{lem:technical_proof_totaro}, we see that $H^1(\GL_N/P,F^*\cEnd(\cE))$ is a trivial representation.  This means that all the assumptions of Theorem~\ref{thm:def_equiv_vb} are satisfied and therefore the forgetful morphism of deformation functors $\Def^{\GL_N}_{F^*\cE} \to \Def_{F^*\cE}$ is smooth as required.

Now, we construct the arrows $\Def^{\GL_N}_{F^*\cE} \longleftarrow \Def_f \longrightarrow \Def_{F_{\GL_d}}$ and prove they are both smooth.  The existence and smoothness of the arrow facing left is the content of Lemma~\ref{lem:g_equivariant_deform_morphism}.  The right arrow is induced by the inclusion $i \colon \GL_{d,W(k)} \to \tilde P$ of the $d \times d$ block.  Its smoothness clearly follows because $i$ is in fact a splitting of $\pi$.

The final claim is a direct corollary of Theorem~\ref{thm:defos_frob_red_gps}.

\end{proof}

\begin{cor}
Let $X$ be a Fano variety violating Kodaira vanishing described in \S\ref{ss:intro_examples_totaro}.  Then $X$ does not lift to any ring where $p \neq 0$.
\end{cor}
\begin{proof}
Recalling the description in \S\ref{ss:intro_examples_totaro}, it suffices to apply Theorem~\ref{thm:defos_flag_varieties} for $d = 2$ and $N = p+2$.
\end{proof}


\section{Another proof of non-liftability of Frobenius}
\label{sec:another_proof_non_liftability}

In this supplementary chapter we present another proof of mod $p^2$ and characteristic zero non-liftability of the Frobenius homomorphism of reductive groups.  We decided to include it in the paper since it is very different in spirit and might be of interest for liftability questions for other, potentially non-reductive groups.


\subsection{Frobenius liftings of $G$ and Frobenius pullbacks of principal $G$-bundles}
\label{sub:criterion_principal_bundles}

This subsection is devoted to the formalization of the observation that a Frobenius lifting of a group $G$ gives a natural way of lifting Frobenius pullbacks of principal $G$-bundles and vector bundles with an appropriate reduction of the structure group.  

\begin{prop}
\label{prop:frobenius_liftings_and_principal_bundles}
Let $\tilde G$ be a group scheme over $W(k)$.  Assume there exists a lifting $F_{G,A} \colon G_A \to G'_A$ of the Frobenius homomorphism of $G$ over $A \in \Art_{W(k)}(k)$.  Then the following assertions hold
\begin{enumerate}
\item for every $G_A$-bundle $P_A \to X_A$ lifting a $G$-bundle $P \to X$ there exists a $G'_A$-bundle $P'_A$ lifting the Frobenius pullback $P' = F_X^*P$, 
\item for every homomorphism $\tilde \pi \colon \tilde G \to \GL_{n,W(k)}$ and a vector bundle $\cE_A$ of rank $n$ on $X_A$, admitting a reduction of structure group along $\pi_A$, the vector bundle $\cE' = F^*\cE$ lifts over $A$,
\end{enumerate}
\end{prop}
\begin{proof}
First, we observe that using considerations of \S\ref{sub:preliminaries_stacks_equivariant_sheaves} the homomorphism $F_{G,A}$ induces a morphism of classifying stacks $BF_{G,A} \colon BG_A \to BG'_A$.  In our setting, we obtain a diagram:
\[
\xymatrix{
  X \ar[rr]^{m_{P}} \ar[d] & & BG \ar[rr]^{BF_G} \ar[d] & & BG' \ar[d] \\
  X_A \ar@{-->}@/_1.2pc/[rrrr]_{m_{P'_A}} \ar[rr]^(0.4){m_{P_A}} & & BG_A \ar[rr]^{BF_{G,A}} & & BG'_A,
}
\]
where $m_{P} \colon X \to BG$ and $m_{P_A} \colon X_A \to BG_A$ are the natural moduli maps associated with the principal bundles $P \to X$ and $P_A \to X_A$.  The composition $BF_{G,A} \circ m_{P_A}$ gives rise to a principal bundle $P'_A$ which by commutativity of the above diagram clearly lifts $P'$.  This finishes the first part of the proof.  For the rest, we extend the diagram above with the reduction of the structure group morphism induced by $\pi_A$ to obtain 

\[
\xymatrix{
  X \ar@/^1.5pc/[rrrr]^{m_{\cE}} \ar[rr]\ar[d] & & BG \ar[rr]^{B\pi}\ar[d] & & BGL_n \ar[rr]^{BF_{\GL_n}} & & BGL_n \ar[d] \\
  X_A \ar@{-->}@/_1.6pc/[rrrrrr]_{m_{\cE'_A}}\ar[rr]^{m_{\cE_A}} & & BG_A \ar[rr]^{BF_{G,A}} & & BG'_A \ar[rr]^{B\pi'_A} & & BGL_{n,A}.
}
\]
Now, we use similar reasoning as above to see that the moduli map $m_{\cE'_A} = B\pi_A \circ BF_{G,A} \circ m_{\cE_A}$ gives a lifting of $\cE'$
\end{proof}

\begin{cor}
\label{cor:frobenius_liftings_and_principal_bundles_formally}
Let $\tilde G$ be a group scheme over $W(k)$.  Assume there exists a formal lifting $\hat{F_G} \colon \hat{G} \to \hat{G}$ of the Frobenius homomorphism of $G$ over a complete $W(k)$-algebra $\hat{A}$.  Then the following assertions hold
\begin{enumerate}
\item for every formal $\hat{G}$-bundle $\hat{P} \to \hat{X}$ lifting a $G$-bundle $P \to X$ there exists a $\hat{G}$-bundle $\hat{P'}$ lifting the Frobenius pullback $P' = F_X^*P$,
\item for every homomorphism $\tilde \pi \colon \tilde G \to \GL_{n,W(k)}$ and a vector bundle $\hat{\cE}$ of rank $n$ on $\hat{X}$, admitting a reduction of structure group along $\hat{\pi}$, the vector bundle $\cE' = F^*\cE$ lifts formally over $\hat{A}$.
\end{enumerate}
\end{cor}
\begin{proof}
We simply use the previous result to derive a compatible system of liftings.
\end{proof}

\begin{remark}
It is important to note that the scheme $X_A$ does not necessarily admit a Frobenius lifting and hence we cannot constuct $P'_A$ as a pullback of $P_A$.
\end{remark}

\begin{example}
In the case of vector bundles, that is, principal $\GL_n$ bundles, the above application of classifying stacks can be substituted with the following more down to earth observation.  Suppose we are given a scheme $X$ over $k$ together with an $A$-lifting $\tilde X$, for $A \in \Art_{W(k)}(k)$.  Let $\cE$ be a vector bundle on $X$, and let $\tilde \cE$ be its lifting over $\tilde X$.  Assume that $\tilde \cE$ is defined by a covering $\tilde U_i$ and a cocycle $\tilde g_{ij} \in \GL_n(\cO_{\tilde U_{ij}})$.  If $\tilde F \colon \GL_{n,A} \to \GL_{n,A}$ is a lifting of the Frobenius homomorphism, then $\tilde F(\tilde g_{ij})$ is the cocycle inducing a lifting of $F^*\cE$.
\end{example}


\subsection{Functoriality properties of Frobenius liftings}
\label{sub:def_frob_obstruction_thy}

In this section, we consider a surjective morphism of group schemes $\tilde \pi \colon \tilde G \to \tilde H$ flat over $W(k)$, and attempt to understand the relation between the Frobenius liftings of $G$ and $H$.  In particular, we prove some necessary criteria for liftings of the Frobenius homomorphism of $G$ to descend along $\pi$, and liftings of the Frobenius homomorphism of $H$ to lift along $\pi$.

\begin{prop}[Descend along morphism with linearly reductive kernels]
\label{prop:descend_along_linearly_reductive}
Assume that the kernel of $\tilde \pi \colon \tilde G \to \tilde H$ is linearly reductive.  Then for every $A \in \Art_{W(k)}(k)$ and a Frobenius lifting $F_{G,A} \colon G_A \to G'_A$ there exists a functorial $F_{H,A} \colon H_A \to H'_A$ compatible with $\pi$.  In particular, there exists a morphism of deformation functors $\Def_{\tilde G} \to \Def_{\tilde H}$
\end{prop}
\begin{proof}
Let $\tilde K$ be the kernel of $\tilde \pi$.  For every $A \in \Art_{W(k)}(k)$ and a Frobenius lifting $F_{G,A} \colon G_A \to G'_A$, we have a diagram:
\[
\xymatrix{
0 \ar[r] & K_A \ar[r]^i & G_A \ar[r]^{\pi_A}\ar[d]^{F_{G,A}} & H_A \ar[r] & 0 \\
0 \ar[r] & K'_A \ar[r]_i & G'_A \ar[r]_{\pi_A} & H'_A \ar[r] & 0.
}
\]
In order to show that $F_{G,A}$ descends to a Frobenius lifting $F_{H,A} \colon H_A \to H'_A$ compatible with $\pi$ we need to show that the obstruction class 
\[
\sigma = \pi_A \circ F_{G,A} \circ i \colon K \to H
\] 
is a trivial homomorphism.  However, reducing the diagram over $k$ and using the commutativity of $\pi$ and Frobenii, we see that $\sigma$ is in fact a deformation of a trivial homomorphism.  By \cite{illusieII}, the infinitesimal deformations of the trivial homomorphism $K \to H$ are classified by $H^1(K,\Lie{h})$ and therefore, using the assumption of linearly reductivity, we see that $\sigma$ is trivial.  This finishes the proof.
\end{proof}

\begin{cor}
For every morphism of reductive group schemes $\pi \colon G \to H$ defined over $W(k)$ with linearly reductive kernel, we have a morphism of deformation functors $\Def_{F_{G}} \to \Def_{F_{H}}$
\end{cor}

\begin{example}\label{example:multiplication_linearly_reductive}
Let $\tilde \pi \colon \SL_n \times \GG_m \to \GL_n$ be a morphism of multiplication by constant diagonal matrices.  Its kernel is isomorphic to the linearly reductive group scheme $\mu_n$, and therefore we obtain a morphism of deformation functors $\Def_{F_{\SL_n \times \GG_m,k}} \to \Def_{F_{\GL_n,k}}$.
\end{example}

The following proposition is standard, but we decided to include the proof for the sake of completeness.

\begin{prop}[Lifting along \'etale maps]
\label{prop:lifting_along_etale_maps}
Assume that $\tilde \pi \colon \tilde G \to \tilde H$ is \'etale.  Then for every $A \in \Art_{W(k)}(k)$ and a  Frobenius lifting $F_{H,A} \colon H_A \to H'_A$ there exists a functorial $F_{G,A} \colon G_A \to G'_A$ compatible with $\pi$.  In particular, there exists a natural morphism of deformation functors $\Def_{F_{\tilde H}} \to \Def_{F_{\tilde G}}$.
\end{prop}
\begin{proof}
We consider the base change $G^{F_{H,A}} = G_A \times_{H_A,F_{H,A}} H_A$.  We claim it fits in the diagram
\[
\xymatrix{
  G_A \ar@{-->}[rd]\ar@/_1.2pc/[rdd]_{\pi} & & \\
   & G^{F_{H,A}} \ar[r]\ar[d]_{\pi''_A}\ar@{}[dr]|-{\square} & G'_A \ar[d]^{\pi'_A} \\
   &  H_A \ar[r]_{F_A} & H'_A,
}
\]
where the dashed arrow is a unique isomorphism compatible with morphisms $\pi$ and $\pi'$.  Indeed, by the Frobenius invariance of \'etale sites \cite[XIV=XV \S{}1 $n^\circ$2, Pr. 2(c)]{SGA5}, the diagram
\[
\xymatrix{
   G \ar[r]^{F_G}\ar[d]_{\pi} & G' \ar[d]^{\pi'} \\
   H \ar[r]_{F_H} & H',
}
\] is cartesian, and therefore the projection $\pi' \colon G^{F_{H,A}} \to H_A$ is an $A$-lifting of an \'etale map $\pi \colon G \to H$.  By the uniqueness of liftings of \'etale maps it is hence uniquely isomorphic to $G_A \to H_A$.  Consequently, the composition of the dashed isomorphism with the projection $G^{F_{H,A}} \to G_A$ is a compatible $A$-lifting of the Frobenius.
\end{proof}

\begin{cor}
For every \'etale morphism of reductive group schemes $\pi \colon G \to H$ defined over $W(k)$, we have a morphism of deformation functors $\Def_{F_{H}} \to \Def_{F_{G}}$
\end{cor}

\begin{example}
In the setting of Example~\ref{example:multiplication_linearly_reductive}, the kernel $\mu_n$ is \'etale if and only if $n$ and $p$ are coprime.  Therefore, under this condition, there exists a morphism of deformation functors $\Def_{F_{\GL_n,k}} \to \Def_{F_{\SL_n \times \GG_m,k}}$.
\end{example}


\subsection{Frobenius homomorphism of reductive groups}
\label{sub:def_frob_reductive_groups}

Finally, we investigate the deformations of the Frobenius homomorphism of reductive groups.  As described in the introduction \S\ref{ss:intro_description_other_approach}, the proofs are based on the existence of liftable vector bundles, admitting an appropriate reduction of structure group, whose Frobenius pullback is not liftable.  More precisely, we are going to apply Proposition~\ref{prop:frobenius_liftings_and_principal_bundles} to vector bundles and principal bundles constructed from the following example of a liftable vector bundle with non-liftable Frobenius pullback.

\begin{example}
\label{example:lauritzen_rao}
In \cite{lauritzen_rao} Lauritzen and Rao provide a simple example of a $6$-dimensional variety violating Kodaira vanishing.  The construction goes as follows.   Let $V$ be a $4$-dimensional vector space over a field $k$ of characteristic $p$, and let $Y \subset \PP(V) \times \PP(V^\vee)$ be the incidence variety of hyperplanes and lines (identified with hyperplanes in the dual space).  Let $\cS$ be a rank two vector bundle on $Y$ defined as the quotient of the tautological hyperplane bundle by the tautological line bundle.  Note that $\PP(\cS)$ is the full flag variety of $V$.  The example of Lauritzen--Rao is a scheme $X$ defined as the projectivization of the Frobenius pullback of $\cS$, i.e., by the cartesian diagram:

\[
\xymatrix{
  X \isom \PP(F^*\cS) \ar[r]\ar[d] & \PP(\cS) \ar[d] \\
  Y \ar[r]_(0.45){F_Y} & Y.
  }
\]

The authors exhibit an ample line bundle $\cL$ on $X$ such that $H^5(X,\omega_X \otimes \cL) > 0$, and therefore Kodaira vanishing is not satisfied.  Observing that $\dim X = 6$ and using Deligne--Illusie results, this implies that for $p \geq 7$ the variety $X$ does not lift mod $p^2$.  Moreover, by Proposition~\ref{prop:defs_projective_bundles}, we see that the vector bundle $F^*\cS$ is not liftable over $W_2(k)$ either.  It is important to remark that the Euler characteristic satisfies $\chi(X,\omega_X \otimes \cL) > 0$ (cf. \cite[end of page 24]{lauritzen_rao}), and therefore we cannot use a semi-continuity argument to deduce that neither $X$ nor $F^*\cS$ lift to a ramified extension of $W(k)$.
\end{example}

We now proceed to the proof of our main theorem. 
\begin{thm}
\label{thm:frob_reductive_non_liftable}
Let $k$ be an algebraically closed field of characteristic $p \geq 7$, and let $G$ be a reductive group of non-exceptional type defined over $k$.  Then the Frobenius homomorphism of $G$ lifts mod $p^2$ or formally to a ramified extension of $W(k)$ if and only if $G$ is linearly reductive.
\end{thm}
\begin{proof}
First, we prove that all linearly reductive groups admit a lifting of the Frobenius homomorphism.  For this purpose, we observe that by the result of Nagata (see \cite{nagata}) linearly reductive groups are of multiplicative type, and by Example~\ref{example:frobenius_multiplicative_type} those clearly admit Frobenius lifting over any base.  

For the converse implication, we first assume that $G$ is an almost-simple semisimple group of non-exceptional type and prove that it does not admit a lifting of the Frobenius homomorphism mod $p^2$.  By \cite[Chapter 22]{milne_algebraic_groups} such groups are classified up to a central isogeny (surjective homomomorphism with finite central kernel) by irreducible Dynkin diagrams, and therefore we conduct a case by case analysis with respect to the type of the diagram.  For the sake of clarity, we first treat the case of type $A_1$.  All groups of this type admit a multiplicative isogeny onto an adjoint group, which is isomorphic to $\PGL_{2,k}$.  By Proposition~\ref{prop:descend_along_linearly_reductive} it therefore suffices to prove that $\PGL_{2,k}$ does not admit a lifting of the Frobenius homomorphism mod $p^2$.  In order to see this, we apply Proposition~\ref{prop:frobenius_liftings_and_principal_bundles} for the principal bundle associated with the $\PP^1$-bundle $\PP(\cS) \to Y$ from Example~\ref{example:lauritzen_rao}, whose Frobenius pullback does not lift mod $p^2$.  In order to generalize the result we use the following

\begin{lemma}
\label{lem:lemma_quot_schemes}
Let $\tilde X$ be a projective scheme over $W(k)$, and let $\cE$ and $\cE'$ be two vector bundles on $X = \tilde X \otimes_{W(k)} k$.  Suppose $\cO_{\tilde X}(1)$ is a relatively ample line bundle on $\tilde X$.  Then for every pair of integers $N$ and $N'$ such that $N + N'$ is sufficiently large and every $A \in \Art_{W(k)}(k)$, a lifting $\cF_A$ of $\cF = \cE(N)\oplus \cE'(-N')$ over $X_A$ induces potentially non-unique deformations of both $\cE$ and $\cE'$.
\end{lemma}
\begin{proof}[Proof of Lemma~\ref{lem:lemma_quot_schemes}]
For the proof we use the obstruction theory for the Quot functor given in \cite[Lemma 2.2.6]{huybrechts_lehn}.  Let $\pi \colon F \to \cE(N)$ be the natural projection map with kernel $\cE'(-N')$.  Given the lifting $\cF_A$ we can consider the functor ${\rm Quot}^{\cF_A}_{X_A/A}$ parametrizing quotients of $\cF_A$.  The map $\pi$ is a $k$-point of this functor.  Using the obstruction theory in \emph{loc. cit} we see that $\pi$ can be extended to a projection $\pi_A \colon \cF_A \to \cQ$ if the group $\Ext^1(\cE'(-N',\cE(N))$ is zero.  This is clearly true for $N+N'$ sufficiently large by Serre vanishing, and hence an extension $\pi_A$ exists (not necessarily unique).  It is easy to see that $\cQ$ is a deformation of $\cE(N)$ and $\Ker \pi_A$ is a deformation of $\cE'(-N')$.  Twisting by appropriate powers of the given ample line bundle we obtain necessary deformations.
\end{proof}

Equipped with the above, we are ready to proceed to other types of Dynkin diagrams.  We begin with the case of $A_n$.  We consider the variety $Y$ from Example~\ref{example:lauritzen_rao} with its natural projective $W(k)$-lifting $\tilde Y$ and the associated ample line bundle $\cO_{\tilde Y}(1)$.  By Lemma~\ref{lem:lemma_quot_schemes} for every $n \geq 2$ there exists an integer $N$ such that the vector bundle $\cS_{n} = \cS \oplus \cO_Y(N)^{\oplus n-2}$ is $p^2$ liftable but its Frobenius pullback is not.  By Proposition~\ref{prop:defs_projective_bundles} we see that the associated projective bundle satisfies analogous properties and therefore applying Proposition~\ref{prop:frobenius_liftings_and_principal_bundles} we see that $\PGL_{n,k}$ does not admit a lifting of the Frobenius homomorphism.  This settles the case of $A_n$ since $\PGL_{n,k}$ is a simple representative \cite[Chapter 22]{milne_algebraic_groups} of the respective Dynkin diagram, i.e., admits an multiplicative isogeny from any other semisimple group of this Dynkin diagram.

For Dynkin diagrams of types $B_n$, $C_n$ and $D_n$ we utilize a similar technique.  More precisely, for $B_n$ we reason as follows.  First, we recall that $\SO_{2n+1}$ is a non-simple representative of this Dynkin diagram.  By Proposition~\ref{prop:frobenius_liftings_and_principal_bundles}(2), in order to prove that $\SO_{2n+1}$ does not admit a lifting of the Frobenius homomorphism it suffices to exhibit a vector bundle of rank $2n+1$ which admits a non-degenerate bilinear pairing, is liftable, but its Frobenius pullback is not.  We claim that for $N$ large enough the bundle $\cS^{\rm sym}_n = \cS_n(N) \oplus \cS_n^\vee(-N) \oplus \cO_Y$ on $Y$ is appropriate.  Indeed, by Lemma~\ref{lem:lemma_quot_schemes}, we see that for $N$ large enough the vector bundle $\cS^{\rm sym}_n$ is liftable mod $p^2$ but its Frobenius pullback is not.  Moreover, the pairing $\mu \colon \cS^{\rm sym}_n \otimes \cS^{\rm sym}_n \to \cO_Y$ given by the matrix

\[
    \kbordermatrix{ & \cS_n(N) & \cS_n^{\vee}(-N) & \cO_Y \\
      \cS_n(N) & 0 & {\rm id}_{\cS_n^{\vee}(-N)} & 0 \\
      \cS_n^{\vee}(-N) & {\rm id}_{\cS_n(N)} & 0 & 0 \\
      \cO_Y & 0 & 0 & {\rm id}_{\cO_Y} }
\]
is non-degenerate.  In order to descend non-liftability to a simple representative, we recall that the index of the Dynkin diagram $B_n$ is smaller than $p$ and therefore the natural isogeny $\SO_{2n+1} \to \SO_{2n+1}^{\rm ad}$ is \'etale and therefore we may apply Proposition~\ref{prop:lifting_along_etale_maps}. 
 
The cases of $C_n$ and $D_n$ are settled analogously by substituting $\SO_{2n+1}$ and its natural representation with $\Sp_{2n}$ and $\SO_{2n}$.  Note that again the indices of those Dynkin diagrams are smaller than $p$ and therefore we may apply Proposition~\ref{prop:lifting_along_etale_maps} without any reservations.  We leave the case of a general reductive group of non-exceptional type as a simple exercise in application of Proposition~\ref{prop:descend_along_linearly_reductive} and Proposition~\ref{prop:lifting_along_etale_maps} based on the classification of reductive groups.


The part of the theorem concerning formal liftability of Frobenius homomorphism of reductive groups over a ramified extension of $W(k)$ is proven analogously.  However, instead of Example~\ref{example:lauritzen_rao} and Proposition~\ref{prop:frobenius_liftings_and_principal_bundles}, we use the projective bundles constructed by Totaro (described in \S\ref{ss:intro_examples_totaro}) and Corollary~\ref{cor:frobenius_liftings_and_principal_bundles_formally}.  To justify this change in the argument, we recall that unlike Lauritzen--Rao examples, those of Totaro clearly do not deform to characteristic zero since they admit a liftable ample line bundle of negative Euler characteristic of a twist by the canonical bundle (cf. end of Example~\ref{example:lauritzen_rao}).
\end{proof}


\bibliographystyle{amsalpha} 
\bibliography{bib.bib}

\renewcommand{\refname}{\rule{2cm}{0.4pt}}
\renewcommand{\addcontentsline}[3]{}


\end{document}

%% file: macrosMath.tex
\usepackage{tikz}
\usepackage{tikz-cd}
\usepackage{todonotes}
\usetikzlibrary{decorations.markings}
\tikzset{degil/.style={
	decoration={markings,
	mark= at position 0.5 with {
	\node[transform shape] (tempnode) {$\backslash$};
	\draw[thick] (tempnode.north east) -- (tempnode.south west);
	}}, postaction={decorate}
}}
\usetikzlibrary{decorations.text, calc, matrix, arrows, decorations.pathreplacing}

\usepackage{kbordermatrix}
\renewcommand{\kbldelim}{(}
\renewcommand{\kbrdelim}{)}

\newtheorem{thm}{Theorem}[section]
\newtheorem{lemma}[thm]{Lemma}
\newtheorem{prop}[thm]{Proposition}
\newtheorem{cor}[thm]{Corollary}

\theoremstyle{definition} 
\newtheorem{defin}[thm]{Definition}
\newtheorem{remark}[thm]{Remark}

\newtheorem{example}[thm]{Example}

\newcommand{\cE}{\mathscr E}
\newcommand{\cF}{\mathscr F}
\newcommand{\cG}{\mathscr G}

\newcommand{\cL}{\mathscr L}

\newcommand{\cO}{\mathscr O}

\newcommand{\cR}{\mathscr R}
\newcommand{\cS}{\mathscr S}
\newcommand{\cQ}{\mathscr Q}
\newcommand{\cT}{\mathscr T}

\newcommand{\cX}{\mathscr X}
\newcommand{\cY}{\mathscr Y}

\newcommand{\bb}[1]{\mathbf{#1}} 
\newcommand{\FF}{\bb{F}}
\newcommand{\NN}{\bb{N}}
\newcommand{\PP}{\bb{P}}

\newcommand{\RR}{\bb{R}}
\newcommand{\ZZ}{\bb{Z}}
\newcommand{\Gm}{\mathbf{G}_m}

\renewcommand{\phi}{\varphi}
\renewcommand{\leq}{\leqslant}
\renewcommand{\geq}{\geqslant}
\newcommand{\isom}{\simeq}

\newcommand{\fppf}{{\rm fppf}}
\newcommand{\cat}[1]{{\normalfont\textbf{#1}}}  
\newcommand{\ra}{\longrightarrow}    

\newcommand{\scalar}[1]{{\langle #1 \rangle}}

\newcommand{\Lie}[1]{\mathfrak{#1}}

\DeclareMathOperator{\Art}{\cat{Art}}
\DeclareMathOperator{\Sch}{\cat{Sch}}

\DeclareMathOperator{\Aut}{Aut}

\DeclareMathOperator{\Def}{Def}
\DeclareMathOperator{\Ext}{Ext}
\DeclareMathOperator{\Gr}{Gr}
\DeclareMathOperator{\Hom}{Hom}

\DeclareMathOperator{\Spec}{Spec}

\DeclareMathOperator{\GL}{GL}
\DeclareMathOperator{\PGL}{PGL}
\DeclareMathOperator{\SL}{SL}

\DeclareMathOperator{\SO}{SO}
\DeclareMathOperator{\Sp}{Sp}

\DeclareMathOperator{\GG}{\mathbb{G}}
\DeclareMathOperator{\Ob}{Ob}
\DeclareMathOperator{\Tan}{T}
\DeclareMathOperator{\Ker}{Ker}

\newcommand{\cExt}{{\mathscr E}\kern -.5pt xt}
\newcommand{\cHom}{\mathscr{H}\kern -.5pt om}
\newcommand{\cEnd}{{\mathscr E}\!nd}

\newcommand{\stacksproj}[1]{{\cite[Tag~\href{http://stacks.math.columbia.edu/tag/#1}{#1}]{stacks-project}}}

\newcommand{\maciek}[1]{}

%% file: defEquivariantVB.bbl
\newcommand{\etalchar}[1]{$^{#1}$}
\providecommand{\bysame}{\leavevmode\hbox to3em{\hrulefill}\thinspace}
\providecommand{\MR}{\relax\ifhmode\unskip\space\fi MR }
\providecommand{\MRhref}[2]{%
  \href{http://www.ams.org/mathscinet-getitem?mr=#1}{#2}
}
\providecommand{\href}[2]{#2}
\begin{thebibliography}{FGI{\etalchar{+}}05}

\bibitem[Bha18]{bhatt}
Bhargav Bhatt, \emph{Non-liftability of vector bundles to witt vectors}, 2018,
  Note on author's personal webpage.

\bibitem[Bui17]{buium2}
Alexandru Buium, \emph{Foundations of arithmetic differential geometry},
  Mathematical surveys and monographs, vol. volume 222, American Mathematical
  Society, Providence Rhode Island, 2017 (eng).

\bibitem[Dem77]{demazure}
M.~Demazure, \emph{Automorphismes et d\'eformations des vari\'et\'es de
  {B}orel}, Invent. Math. \textbf{39} (1977), no.~2, 179--186. \MR{0435092}

\bibitem[DI87]{deligne_illusie}
Pierre Deligne and Luc Illusie, \emph{Rel\`evements modulo {$p^2$} et
  d\'ecomposition du complexe de de {R}ham}, Invent. Math. \textbf{89} (1987),
  no.~2, 247--270. \MR{894379}

\bibitem[Eke03]{Ekedahl}
Torsten Ekedahl, \emph{On non-liftable {C}alabi-{Y}au threefolds}, 2003.

\bibitem[FGI{\etalchar{+}}05]{FGA}
Barbara Fantechi, Lothar G\"ottsche, Luc Illusie, Steven~L. Kleiman, Nitin
  Nitsure, and Angelo Vistoli, \emph{Fundamental algebraic geometry},
  Mathematical Surveys and Monographs, vol. 123, American Mathematical Society,
  Providence, RI, 2005, Grothendieck's FGA explained. \MR{2222646}

\bibitem[FM98]{fantechi_manetti}
Barbara Fantechi and Marco Manetti, \emph{Obstruction calculus for functors of
  {A}rtin rings. {I}}, J. Algebra \textbf{202} (1998), no.~2, 541--576.
  \MR{1617687}

\bibitem[Har10]{hartshorne_deformation}
Robin Hartshorne, \emph{Deformation theory}, Graduate Texts in Mathematics,
  vol. 257, Springer, New York, 2010.
  \MR{\href{http://www.ams.org/mathscinet-getitem?mr=2583634}{2583634}}

\bibitem[Hir99]{Hirokado}
Masayuki Hirokado, \emph{A non-liftable {C}alabi-{Y}au threefold in
  characteristic {$3$}}, Tohoku Math. J. (2) \textbf{51} (1999), no.~4,
  479--487. \MR{1725623}

\bibitem[HL10]{huybrechts_lehn}
Daniel Huybrechts and Manfred Lehn, \emph{The geometry of moduli spaces of
  sheaves}, second ed., Cambridge Mathematical Library, Cambridge University
  Press, Cambridge, 2010. \MR{2665168}

\bibitem[Ill71]{illusieI}
Luc Illusie, \emph{Complexe cotangent et d\'eformations. {I}}, Lecture Notes in
  Mathematics, Vol. 239, Springer-Verlag, Berlin-New York, 1971. \MR{0491680}

\bibitem[Ill72]{illusieII}
\bysame, \emph{Complexe cotangent et d\'eformations. {II}}, Lecture Notes in
  Mathematics, Vol. 283, Springer-Verlag, Berlin-New York, 1972. \MR{0491681}

\bibitem[Jan03]{jantzen}
Jens~Carsten Jantzen, \emph{Representations of algebraic groups}, second ed.,
  Mathematical Surveys and Monographs, vol. 107, American Mathematical Society,
  Providence, RI, 2003. \MR{2015057}

\bibitem[LR97]{lauritzen_rao}
N.~Lauritzen and A.~P. Rao, \emph{Elementary counterexamples to {K}odaira
  vanishing in prime characteristic}, Proc. Indian Acad. Sci. Math. Sci.
  \textbf{107} (1997), no.~1, 21--25. \MR{1453823}

\bibitem[LS14]{liedtke_satriano}
Christian Liedtke and Matthew Satriano, \emph{On the birational nature of
  lifting}, Adv. Math. \textbf{254} (2014), 118--137. \MR{3161094}

\bibitem[Mil17]{milne_algebraic_groups}
J.~S. Milne, \emph{Algebraic groups}, Cambridge Studies in Advanced
  Mathematics, vol. 170, Cambridge University Press, Cambridge, 2017, The
  theory of group schemes of finite type over a field. \MR{3729270}

\bibitem[Nag62]{nagata}
Masayoshi Nagata, \emph{Complete reducibility of rational representations of a
  matric group}, J. Math. Kyoto Univ. \textbf{1} (1961/1962), 87--99.
  \MR{0142667}

\bibitem[Sch68]{schlessinger}
Michael Schlessinger, \emph{Functors of {A}rtin rings}, Trans. Amer. Math. Soc.
  \textbf{130} (1968), 208--222. \MR{0217093}

\bibitem[Sch04]{Schroer}
Stefan Schr\"oer, \emph{Some {C}alabi-{Y}au threefolds with obstructed
  deformations over the {W}itt vectors}, Compos. Math. \textbf{140} (2004),
  no.~6, 1579--1592. \MR{2098403}

\bibitem[Spr98]{springer}
T.~A. Springer, \emph{Linear algebraic groups}, second ed., Progress in
  Mathematics, vol.~9, Birkh\"auser Boston, Inc., Boston, MA, 1998.
  \MR{1642713}

\bibitem[{Sta}17]{stacks-project}
The {Stacks Project Authors}, \emph{{Stacks Project}},
  \url{http://stacks.math.columbia.edu}, 2017.

\bibitem[Tot17]{totaro_failure}
Burt Totaro, \emph{The failure of {K}odaira vanishing for {F}ano varieties, and
  terminal singularities that are not {C}ohen--{M}acaulay},
  \href{https://arxiv.org/abs/1710.04364}{arXiv:1710.04364} (2017).

\end{thebibliography}


\begin{thebibliography}{plain}

\bibitem[SGA $3_{\rm III}$]{SGA3III}
\emph{Sch\'emas en groupes ({SGA} 3). {T}ome {III}. {S}tructure des sch\'emas en groupes r\'eductifs}, 
  Documents Math\'ematiques (Paris) [Mathematical Documents (Paris)], 3, 
  Soci\'et\'e Math\'ematique de France, Paris, 2011,
  S\'eminaire de G\'eom\'etrie Alg\'ebrique du Bois Marie 1962--64,
  [Algebraic Geometry Seminar of Bois Marie 1962--64], 
  A seminar directed by M. Demazure and A. Grothendieck with the collaboration 
  of M. Artin, J.-E. Bertin, P. Gabriel, M. Raynaud and J-P. Serre, 
  Revised and annotated edition of the 1970 French original,
  \MR{\href{http://www.ams.org/mathscinet-getitem?mr=2867622}{2867622}}
  
\bibitem[SGA 5]{SGA5}
\emph{Cohomologie {$\ell$}-adique et fonctions {$L$}}, Lecture Notes in 
  Mathematics, Vol. 589, Springer-Verlag, Berlin, 1977, S{\'e}minaire de 
  G{\'e}ometrie Alg{\'e}brique du Bois-Marie 1965--1966 (SGA 5), Edit{\'e} par Luc Illusie. 
  \MR{\href{http://www.ams.org/mathscinet-getitem?mr=0491704}{0491704} (58 \#10907)}

\end{thebibliography}
